\documentclass[11pt,a4paper]{article}
\usepackage{epsf,epsfig,amsfonts,amsgen,amsmath,amstext,amsbsy,amsopn,amsthm,lineno}
\usepackage{color}
\usepackage{ebezier,eepic}
\newtheorem{theorem}{Theorem}

\newtheorem{lemma}{Lemma}
\newtheorem{false statement}{False statement}

\theoremstyle{definition}

\newtheorem{claim}{Claim}

\newtheorem{case}{Case}

\newcommand{\de}{{\rm def}}

\begin{document}

\title{\bf\Large On traceability of claw-$o_{-1}$-heavy graphs
\thanks{Supported by NSFC (No.~11271300) and the Doctorate Foundation of Northwestern
Polytechnical University (No. cx201202)}}

\date{}

\author{Binlong Li and Shenggui Zhang\thanks{Corresponding
author. E-mail address: sgzhang@nwpu.edu.cn (S. Zhang).}\\[2mm]
\small Department of Applied Mathematics,
\small Northwestern Polytechnical University,\\
\small Xi'an, Shaanxi 710072, P.R.~China} \maketitle

\begin{abstract}
A graph is called traceable if it contains a Hamilton path, i.e., a
path passing through all its vertices. Let $G$ be a graph on $n$
vertices. $G$ is called claw-$o_{-1}$-heavy if every induced claw
($K_{1,3}$) of $G$ has a pair of nonadjacent vertices with degree
sum at least $n-1$ in $G$. In this paper we show that a
claw-$o_{-1}$-heavy graph $G$ is traceable if we impose certain
additional conditions on $G$ involving forbidden induced subgraphs.

\smallskip
\end{abstract}

\section{Introduction}

We use Bondy and Murty \cite{Bondy_Murty} for terminology and
notation not defined here and consider finite simple graphs only.

A graph $G$ is \emph{traceable} if it contains a \emph{Hamilton
path}, i.e., a path containing all vertices of $G$; and it is
\emph{hamiltonian} if it contains a \emph{Hamilton cycle}, i.e., a
cycle containing all vertices of $G$.

Here we first shortly describe some types of sufficient conditions
for the existence of Hamilton cycles, one of which have been popular
research areas for a considerable time, namely \emph{forbidden
subgraph conditions}. Before we do so, we need to introduce some
additional terminology.

Let $G$ be a graph. If a subgraph $G'$ of $G$ contains all edges
$xy\in E(G)$ with $x,y\in V(G')$, then $G'$ is called an
\emph{induced subgraph} of $G$ (or a subgraph of $G$ induced by
$V(G')$). For a given graph $H$, we say that $G$ is \emph{$H$-free}
if $G$ does not contain an induced subgraph isomorphic to $H$. For a
family $\mathcal{H}$ of graphs, $G$ is called
\emph{$\mathcal{H}$-free} if $G$ is $H$-free for every
$H\in\mathcal{H}$. If $G$ is $H$-free, then $H$ is called a
\emph{forbidden subgraph} of $G$.

The graph $K_{1,3}$ is called a \emph{claw}, in which the only
vertex of degree 3 is called the \emph{center} and the other
vertices are the \emph{end vertices}.

Forbidden subgraph conditions for hamiltonicity have been known
since the early 1980s, but Bedrossian was the first to study the
characterization of all pairs of forbidden graphs for hamiltonicity
in his PhD thesis of 1991 \cite{Bedrossian}. Before we state his
result, we first note that forbidding $K_1$ is absurd because we
always assume a graph has a nonempty vertex set. Moreover, we note
that a $K_2$-free graph is an empty graph (contains no edges), so it
is trivially non-hamiltonian. In this paper, we therefore assume
that all the forbidden subgraphs we will consider have at least
three vertices. Finally, we note that every connected $P_3$-free
graph is complete, and then is trivially hamiltonian (if it has at
least 3 vertices), and it is in fact easy to show that $P_3$ is the
only connected graph $S$ such that every 2-connected $S$-free graph
is hamiltonian. The next result of Bedrossian deals with pairs of
forbidden subgraphs, excluding $P_3$.

\begin{theorem}[Bedrossian \cite{Bedrossian}]
Let $R$ and $S$ be connected graphs with $R,S\neq P_{3}$ and let $G$
be a 2-connected graph. Then $G$ being $\{R,S\}$-free implies $G$ is
hamiltonian if and only if (up to symmetry) $R=K_{1,3}$ and $S=P_4$,
$P_5$, $P_6$, $C_3$, $Z_1$, $Z_2$, $B$, $N$ or $W$.
\end{theorem}

\begin{center}
\begin{picture}(360,200)
\thicklines

\put(5,140){\multiput(20,30)(50,0){5}{\put(0,0){\circle*{6}}}
\put(20,30){\line(1,0){100}} \put(170,30){\line(1,0){50}}
\qbezier[4](120,30)(145,30)(170,30) \put(18,35){$v_1$}
\put(68,35){$v_2$} \put(118,35){$v_3$} \put(168,35){$v_{i-1}$}
\put(218,35){$v_i$} \put(115,10){$P_i$}}

\put(265,125){\put(20,30){\circle*{6}} \put(70,30){\circle*{6}}
\put(45,55){\circle*{6}} \put(20,30){\line(1,0){50}}
\put(20,30){\line(1,1){25}} \put(70,30){\line(-1,1){25}}
\put(40,10){$C_3$}}

\put(0,0){\put(20,30){\circle*{6}} \put(70,30){\circle*{6}}
\multiput(45,55)(0,25){4}{\circle*{6}} \put(20,30){\line(1,0){50}}
\put(20,30){\line(1,1){25}} \put(70,30){\line(-1,1){25}}
\put(45,55){\line(0,1){25}} \put(45,105){\line(0,1){25}}
\qbezier[4](45,80)(45,92.5)(45,105) \put(50,78){$v_1$}
\put(50,103){$v_{i-1}$} \put(50,128){$v_i$} \put(40,10){$Z_i$}}

\put(90,0){\put(45,40){\circle*{6}} \put(45,40){\line(-1,1){25}}
\put(45,40){\line(1,1){25}} \put(20,65){\line(1,0){50}}
\multiput(20,65)(50,0){2}{\multiput(0,0)(0,30){2}{\put(0,0){\circle*{6}}}
\put(0,0){\line(0,1){30}}} \put(25,10){$B$ (Bull)}}

\put(180,0){\multiput(20,30)(50,0){2}{\multiput(0,0)(0,30){2}{\put(0,0){\circle*{6}}}
\put(0,0){\line(0,1){30}}}
\multiput(45,85)(0,30){2}{\put(0,0){\circle*{6}}}
\put(45,85){\line(0,1){30}} \put(20,60){\line(1,0){50}}
\put(20,60){\line(1,1){25}} \put(70,60){\line(-1,1){25}}
\put(25,10){$N$ (Net)}}

\put(270,0){\put(45,30){\circle*{6}} \put(20,55){\line(1,0){50}}
\put(45,30){\line(1,1){25}} \put(45,30){\line(-1,1){25}}
\multiput(20,55)(0,30){2}{\put(0,0){\circle*{6}}}
\multiput(70,55)(0,30){3}{\put(0,0){\circle*{6}}}
\put(20,55){\line(0,1){30}} \put(70,55){\line(0,1){60}}
\put(10,10){$W$ (Wounded)}}

\end{picture}

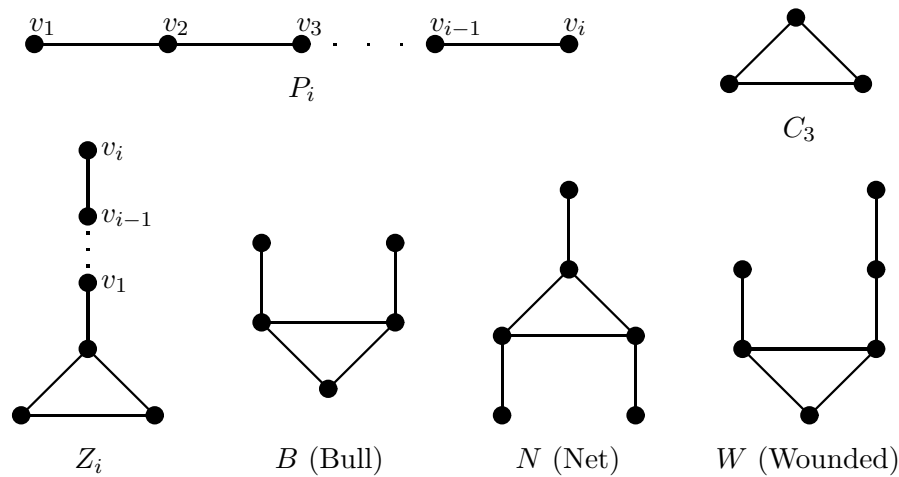
\small{Fig. 1. Graphs $P_i,C_3,Z_i,B,N$ and $W$}
\end{center}

A well-known sufficient condition for a graph to be hamiltonian was
given by Ore \cite{Ore} in 1960, and was called \emph{degree sum
condition}. It states that a graph $G$ on $n\ge 3$ vertices is
hamiltonian if every pair of nonadjacent vertices of $G$ has degree
sum at least $n$.

In an earlier paper \cite{Li_Ryjacek_Wang_Zhang}, we combine the two
types of conditions, i.e., to restrict the degree sum condition to
certain subgraphs, to obtain a new type of conditions for
hamiltonicity that we generally address as \emph{heavy subgraph
conditions}. Before we present the results of it, we need a few more
definitions.

Let $G$ be a graph on $n$ vertices, and let $G'$ be an induced
subgraph of $G$. We say that $G'$ is \emph{heavy} in $G$ if there
are two nonadjacent vertices in $V(G')$ with degree sum at least $n$
in $G$. For a given fixed graph $H$, the graph $G$ is called
\emph{$H$-heavy} if every induced subgraph of $G$ isomorphic to $H$
is heavy. For a family $\mathcal{H}$ of graphs, $G$ is called
\emph{$\mathcal{H}$-heavy} if $G$ is $H$-heavy for every
$H\in\mathcal{H}$. Note that an $H$-free graph is also $H$-heavy;
and if $H_1$ is an induced subgraph of $H_2$, then an $H_1$-free
($H_1$-heavy) graph is also $H_2$-free ($H_2$-heavy).

For hamiltonicity we obtained the following counterpart of
Bedrossian's Theorem (it was also shown in
\cite{Li_Ryjacek_Wang_Zhang} that the only connected graph $S$ such
that every 2-connected $S$-free graph is hamiltonian is $P_3$).

\begin{theorem}[Li et al. \cite{Li_Ryjacek_Wang_Zhang}]
Let $R$ and $S$ be connected graphs with $R,S\neq P_3$ and let $G$
be a 2-connected graph. Then $G$ being $\{R,S\}$-heavy implies $G$
is hamiltonian if and only if (up to symmetry) $R=K_{1,3}$ and
$S=P_4$, $P_5$, $C_3$, $Z_1$, $Z_2$, $B$, $N$ or $W$.
\end{theorem}

Comparing the two theorems, we note that the claw $K_{1,3}$ is
always one of the heavy pairs, and $P_6$ is the only graph that
appears in the list of Bedrossian's Theorem but is missing here. One
can find an example in \cite{Li_Ryjacek_Wang_Zhang} showing that
$P_6$ has to be excluded in the above theorem.

Now we consider the subgraph conditions for traceability of graphs.
First, as pointed out before, if a graph is connected and
$P_3$-free, then it is a complete graph and of course is traceable.
In fact, $P_3$ is the only connected graph $S$ such that every
connected $S$-free graphs is traceable. The following theorem on
forbidden pair of subgraphs for traceability is well known.

\begin{theorem}[Duffus, Jacobson and Gould \cite{Duffus_Jacobson_Gould}]
If $G$ is a connected $\{K_{1,3},N\}$-free graph, then $G$ is
traceable.
\end{theorem}

Obviously if $H$ is an induced subgraph of $N$, then $\{K_{1,3},H\}$
will also solve this problem. Faudree et al. proved these are the
only forbidden pairs with such property.

\begin{theorem}[Faudree and Gould \cite{Faudree_Gould}]
Let $R$ and $S$ be connected graphs with $R,S\neq P_3$ and let $G$
be a connected graph. Then $G$ being $\{R,S\}$-free implies $G$ is
traceable if and only if (up to symmetry) $R=K_{1,3}$ and
$S=C_3,P_4,Z_1,B$ or $N$.
\end{theorem}

A counterpart of Ore's Theorem shows that every graph on $n$
vertices in which every pair of nonadjacent vertices has degree sum
at least $n-1$, is traceable. The main object of this paper, is to
restrict the degree sum condition to certain subgraphs, to obtain a
new type of conditions for traceability. We first give some
definitions.

Let $G$ be a graph on $n$ vertices and $G'$ an induced subgraph of
$G$. We say that $G'$ is $o_{-1}$-\emph{heavy} if there are two
nonadjacent vertices in $V(G')$ with degree sum at least $n-1$ in
$G$. For a given graph $H$, $G$ is called $H$-$o_{-1}$-\emph{heavy}
if every induced subgraph of $G$ isomorphic to $H$ is
$o_{-1}$-heavy. For a family $\mathcal{H}$ of graphs, $G$ is called
$\mathcal{H}$-$o_{-1}$-\emph{heavy} if $G$ is $H$-$o_{-1}$-heavy for
every $H\in\mathcal{H}$.

In this paper, instead of $K_{1,3}$-free ($K_{1,3}$-heavy,
$K_{1,3}$-$o_{-1}$-heavy), we use the terminology claw-free
(claw-heavy, claw-$o_{-1}$-heavy).

Now we consider the following question: for which graph $S$ (which
pair of graphs $R,S$), a connected graph is $S$-$o_{-1}$-heavy
($\{R,S\}$-$o_{-1}$-heavy) implies it is traceable?

First, we will prove in Section 4 that every connected
$P_3$-$o_{-1}$-heavy graph is traceable.

\begin{theorem}
If $G$ is a connected $P_3$-$o_{-1}$-heavy graph, then $G$ is
traceable.
\end{theorem}

It is not difficult to see that $P_3$ is the only connected graph
$S$ such that every connected $S$-$o_{-1}$-heavy graph is traceable.
It is more interesting to consider which pair of graphs $R$ and $S$
other than $P_3$ imply that every connected $\{R,S\}$-$o_{-1}$-heavy
graph is traceable. In fact, as we show bellow, there is only one
such pair of subgraphs.

\begin{theorem}
Let $R$ and $S$ be connected graphs with $R,S\neq P_3$ and let $G$
be a connected graph. Then $G$ being $\{R,S\}$-$o_{-1}$-heavy
implies $G$ is traceable if and only if (up to symmetry) $R=K_{1,3}$
and $S=C_3$.
\end{theorem}

Since $C_3$ is a clique, a graph is $C_3$-$o_{-1}$-heavy is
equivalent to it is $C_3$-free. Thus for the sufficiency of Theorem
6, we only need to prove that every connected claw-$o_{-1}$-heavy
and $C_3$-free graph is traceable. In fact, we can prove a stronger
theorem as bellow.

\begin{theorem}
If $G$ is a connected claw-$o_{-1}$-heavy and $Z_1$-free graph, then
$G$ is traceable.
\end{theorem}

We postpone the proof of Theorem 7 in Section 5, and in Section 6,
we will prove the following theorem, which shows another subgraph
$S$ such that a connected claw-$o_{-1}$-heavy and $S$-free graph is
traceable.

\begin{theorem}
If $G$ is a connected claw-$o_{-1}$-heavy and $P_4$-free graph, then
$G$ is traceable.
\end{theorem}

In fact, these are the only forbidden subgraphs satisfying such
property.

\begin{theorem}
Let $S$ be connected graphs with $S\neq P_3$ and let $G$ be a
connected claw-$o_{-1}$-heavy graph. Then $G$ being $S$-free implies
$G$ is traceable if and only if $S=C_3,Z_1$ or $P_4$.
\end{theorem}

We prove the necessity of Theorems 4 and 7 in Section 3.

\section{Some preliminaries}

We first give some additional terminology and notation.

Let $G$ be a graph, $P$ be a path of $G$ and $x,y\in V(P)$. We use
$P[x,y]$ to denote the subpath of $P$ from $x$ to $y$.

Let $G$ be a graph on $n$ vertices and $k$ be an integer. We call a
sequence of vertices $P=v_1v_2\cdots v_k$ an \emph{$o_{-1}$-path} of
$G$, if for all $i\in [1,k-1]$, either $v_iv_{i+1}\in E(G)$ or
$d(v_i)+d(v_{i+1})\geq n-1$. The \emph{deficit} of $P$ is defined by
$\de(P)=|\{i\in[1,k-1]: v_iv_{i+1}\notin E(G)\}|$. Thus a path is an
$o_{-1}$-path with deficit 0.

Now, we prove the following lemma on $o_{-1}$-paths.

\begin{lemma}
Let $G$ be a graph and $P$ an $o_{-1}$-path of $G$. Then there
exists a path of $G$ which contains all the vertices in $V(P)$.
\end{lemma}

\begin{proof}
Assume the opposite. Let $P'$ be an $o_{-1}$-path which contains all
the vertices in $V(P)$ such that $\de(P')$ is as small as possible.
Then we have $\de(P')\geq 1$. Without loss of generality, we assume
that $P'=v_1v_2\cdots v_p$ such that $v_kv_{k+1}\notin E(G)$ and
$d(v_k)+d(v_{k+1})\geq n-1$, where $1\leq k\leq p-1$.

If $v_k$ and $v_{k+1}$ have a common neighbor in $V(G)\backslash
V(P)$, denote it by $x$. Then
$P''=P'[v_1,v_k]v_kxv_{k+1}P'[v_{k+1},v_p]$ is an $o_{-1}$-path
which contains all the vertices in $V(P)$ with deficit smaller than
$\de(P')$, a contradiction.

So we assume that $N_{G-P'}(v_1)\cap N_{G-P'}(v_k)=\emptyset$. Then
we have $d_{P'}(v_k)+d_{P'}(v_{k+1})\geq |V(P')|-1$ by
$d(v_k)+d(v_{k+1})\geq n-1$.

If $v_1v_{k+1}\in E(G)$, then
$P''=P'[v_k,v_1]v_1v_{k+1}P'[v_{k+1},v_p]$ is an $o_{-1}$-path which
contains all the vertices in $V(P)$ with deficit smaller than
$\de(P')$, a contradiction. Thus we assume that $v_1v_{k+1}\notin
E(G)$, and similarly, $v_pv_k\notin E(G)$. Thus, there exists
$i\in[1,p-1]\setminus\{k\}$ such that $v_i\in N_P(v_k)$ and
$v_{i+1}\in N_P(v_{k+1})$.

If $1\leq i\leq k-1$, then
$P''=P'[v_1,v_i]v_iv_kP'[v_k,v_{i+1}]v_{i+1}v_{k+1}P'[v_{k+1},v_p]$
is an $o_{-1}$-path which contains all the vertices in $V(P)$ with
deficit smaller than $\de(P')$, a contradiction. If $k+1\leq i\leq
p-1$, then
$P''=P'[v_1,v_k]v_kv_iP'[v_i,v_{k+1}]v_{k+1}v_{i+1}P'[v_{i+1},v_p]$
is an $o_{-1}$-path which contains all the vertices in $V(P)$ with
deficit smaller than $\de(P')$, a contradiction.
\end{proof}

In the following, we use $\widetilde{E}_{-1}(G)$ to denote the set
$\{uv: uv\in E(G)$ or $d(u)+d(v)\geq n-1\}$.

We now give a lemma on claw-$o_{-1}$-heavy graphs.

\begin{lemma}
Let $G$ be a connected claw-$o_{-1}$-heavy graphs and $x$ be a
cut-vertex of $G$. Then\\
(1) $G-x$ contains exactly two components; and\\
(2) if $x_1$ and $x_2$ are two neighbors of $x$ in a common
component of $G-x$, then $x_1x_2\in\widetilde{E}_{-1}(G)$.
\end{lemma}

\begin{proof}
If there are at least three components of $G-x$, then let $H_1, H_2$
and $H_3$ be three components. Let $x_1,x_2$ and $x_3$ be neighbors
of $x$ in $H_1, H_2$ and $H_3$, respectively. Then the subgraph
induced by $\{x,x_1,x_2,x_3\}$ is a claw. Besides, for $1\leq
i<j\leq 3$, $d(x_i)+d(x_j)\leq |V(H_i)|+|V(H_j)|\leq n-2$, a
contradiction. Thus, $G-x$ has exactly two components.

Let $x_1,x_2$ be two neighbors of $x$ in a common component $H$. If
$x_1x_2\notin E(G)$, then let $x'$ be a neighbor of $x$ in the other
component $H'$ and the subgraph induced by $\{x,x_1,x_2,x'\}$ is a
claw. Besides, for $i=1,2$, $d(x_i)+d(x')\leq |V(H)|-1+|V(H')|\leq
n-2$. Since $G$ is claw-$o_{-1}$-heavy, we have $d(x_1)+d(x_2)\geq
n-1$.
\end{proof}

\section{The proof of the necessity of Theorems 6 and 9}

We construct two non-traceable graphs as follows.

\begin{center}
\setlength{\unitlength}{0.75pt}
\begin{picture}(500,260)

\thicklines

\put(0,15){\put(70,80){\thinlines\circle{100}}
\multiput(20,180)(75,0){2}{\multiput(0,0)(25,0){2}{\multiput(0,0)(0,30){2}{\circle*{4}}
\put(0,0){\line(0,1){30}}}} \qbezier[2](65,180)(70,180)(75,180)
\qbezier[2](65,210)(70,210)(75,210) \thinlines
\put(20,180){\line(0,-1){100}} \qbezier(20,180)(120,130)(120,80)
\qbezier(45,180)(20,130)(20,80) \qbezier(45,180)(120,130)(120,80)
\qbezier(95,180)(20,130)(20,80) \qbezier(95,180)(120,130)(120,80)
\qbezier(120,180)(20,130)(20,80) \put(120,180){\line(0,-1){100}}
\put(24,178){$x_1$} \put(49,178){$x_2$} \put(99,178){$x_{k-1}$}
\put(124,178){$x_k$} \put(24,208){$x'_1$} \put(49,208){$x'_2$}
\put(99,208){$x'_{k-1}$} \put(124,208){$x'_k$}
\put(55,75){$K_{n-2k}$}} \put(25,25){$G_1$ ($k\geq 3$ and}
\put(30,10){$n\geq 4k-3$)}

\put(140,0){\put(180,80){\thinlines\circle{100}}

\multiput(130,180)(75,0){2}{\multiput(0,0)(25,0){2}{\circle*{4}}}
\qbezier[2](175,180)(180,180)(185,180)
\multiput(180,210)(0,30){2}{\circle*{4}}
\put(180,210){\line(-5,-3){50}} \put(180,210){\line(-5,-6){25}}
\put(180,210){\line(5,-6){25}} \put(180,210){\line(5,-3){50}}
\put(180,210){\line(0,1){30}} {\thinlines
\put(130,180){\line(0,-1){100}} \qbezier(130,180)(230,130)(230,80)
\qbezier(155,180)(130,130)(130,80)
\qbezier(155,180)(230,130)(230,80)
\qbezier(205,180)(130,130)(130,80)
\qbezier(205,180)(230,130)(230,80)
\qbezier(230,180)(130,130)(130,80) \put(230,180){\line(0,-1){100}}}
\put(134,178){$x_1$} \put(159,178){$x_2$} \put(209,178){$x_{k-1}$}
\put(234,178){$x_k$} \put(184,208){$x'$} \put(184,238){$x$}

\multiput(80,30)(0,75){2}{\multiput(0,0)(0,25){2}{\circle*{4}}}
\qbezier[2](80,75)(80,80)(80,85)
\multiput(50,80)(-30,0){2}{\circle*{4}} \put(50,80){\line(3,-5){30}}
\put(50,80){\line(6,-5){30}} \put(50,80){\line(6,5){30}}
\put(50,80){\line(3,5){30}} \put(50,80){\line(-1,0){30}} {\thinlines
\put(80,30){\line(1,0){100}} \qbezier(80,30)(130,130)(180,130)
\qbezier(80,55)(130,30)(180,30) \qbezier(80,55)(130,130)(180,130)
\qbezier(80,105)(130,30)(180,30) \qbezier(80,105)(130,130)(180,130)
\qbezier(80,130)(130,30)(180,30) \put(80,130){\line(1,0){100}}}
\put(76,134){$y_1$} \put(76,109){$y_2$} \put(76,59){$y_{k-1}$}
\put(76,34){$y_k$} \put(46,84){$y'$} \put(16,84){$y$}

\multiput(280,30)(0,75){2}{\multiput(0,0)(0,25){2}{\circle*{4}}}
\qbezier[2](280,75)(280,80)(280,85)
\multiput(310,80)(30,0){2}{\circle*{4}}
\put(310,80){\line(-3,-5){30}} \put(310,80){\line(-6,-5){30}}
\put(310,80){\line(-6,5){30}} \put(310,80){\line(-3,5){30}}
\put(310,80){\line(1,0){30}} {\thinlines
\put(280,30){\line(-1,0){100}} \qbezier(280,30)(230,130)(180,130)
\qbezier(280,55)(230,30)(180,30) \qbezier(280,55)(230,130)(180,130)
\qbezier(280,105)(230,30)(180,30)
\qbezier(280,105)(230,130)(180,130)
\qbezier(280,130)(230,30)(180,30) \put(280,130){\line(-1,0){100}}}
\put(276,134){$z_1$} \put(276,109){$z_2$} \put(276,59){$z_{k-1}$}
\put(276,34){$z_k$} \put(306,84){$z'$} \put(336,84){$z$}

\put(155,75){$K_{n-3k-6}$} \put(95,10){$G_2$ ($k\geq 5$ and $n\geq
6k+9$)}}

\end{picture}

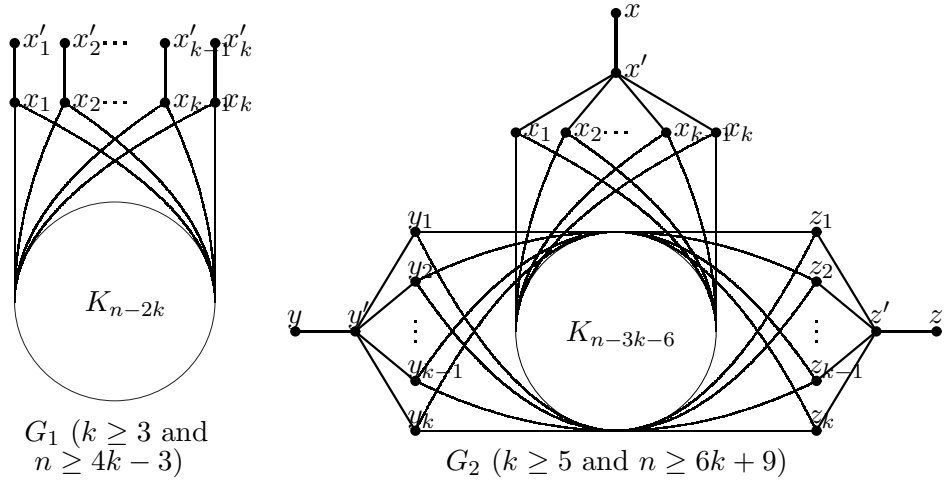
\small Fig. 2. Two non-traceable graphs.
\end{center}

Let $R$ and $S$ be two connected graphs other than $P_3$ such that
every connected $\{R,S\}$-$o_{-1}$-heavy graph is traceable. Then by
Theorem 2, up to symmetry, $R=K_{1,3}$ and $S$ be $C_3,P_4,Z_1,B$ or
$N$. Note that $G_1$ is $\{K_{1,3},P_4\}$-$o_{-1}$-heavy and $G_2$
is $\{K_{1,3},Z_1\}$-$o_{-1}$-heavy. Thus $S$ must be $C_3$.

Let $S$ be a connected graph other than $P_3$ such that every
connected claw-$o_{-1}$-heavy and $S$-free graph is traceable. By
Theorem 2, $S$ must be $C_3,P_4,Z_1,B$ or $N$. Note that $G_1$ is
$B$-free. Thus $S$ must be $C_3,P_4$ or $Z_1$.

\section{Proof of Theorem 5}

We use $n$ to denote the order of $G$. Let $P=v_1v_2\cdots v_p$ be a
longest path of $G$. Assume that $G$ is not traceable. Then
$V(G)\backslash V(P)\neq \emptyset$. Since $G$ is connected, there
exists a vertex $x\in V(G-P)$ joined to $P$. Let $v_i$ be a neighbor
of $x$ in $P$. Clearly $v_i\neq v_p$, otherwise $P'=Pv_px$ is a path
longer than $P$. If $xv_{i+1}\in E(G)$, then
$P'=P[v_1,v_i]v_ixv_{i+1}P[v_{i+1},v_p]$ is a path longer than $P$,
a contradiction. Thus we assume that $xv_{i+1}\notin E(G)$. Since
$G$ is $P_3$-$o_{-1}$-heavy, we have that $d(x)+d(v_{i+1})\geq n-1$.
Thus $P'=P[v_1,v_i]v_ixv_{i+1}P[v_{i+1},v_p]$ is an $o_{-1}$-path of
$G$. By Lemma 2, there is a path of $G$ containing all the vertices
in $P'$, a contradiction.

\section{Proof of Theorem 7}

If $G$ contains only one or two vertices, then the result is
trivially true. So we assume that $G$ contains at least three
vertices. We use $n$ to denote the order of $G$. We distinguish two
cases.

\begin{case}
$G$ is separable.
\end{case}

If $G$ itself is a path, then we have nothing need to prove. Thus we
assume that $G$ is not a path. Thus there must be a cut-vertex of
$G$ with degree at least 3. Let $x$ be such a cut-vertex. By Lemma
1, $G-x$ has exactly two components. Let $C$ and $D$ be the two
components of $G-x$. Since $d(x)\geq 3$, without loss of generality,
we assume that $x$ has at least two neighbors in $D$.

If $x$ contained in a triangle $xx'x''$, then $x'$ and $x''$ is in a
common component of $G-x$. Without loss of generality, Let
$x',x''\in V(D)$. Let $w$ be a neighbor of $x$ in $C$. Then the
subgraph induced by $\{x,x',x'',w\}$ is a $Z_1$, a contradiction.
Thus we assume that $x$ is not contained in a triangle and $N(x)$ is
an independent set.

Let $y$ be a neighbor of $x$. If $y$ contained in a triangle
$yy'y''$, then clearly $xy',xy''\notin E(G)$, otherwise $x$ will be
contained in a triangle. Thus the subgraph induced by
$\{y,y',y'',x\}$ is a $Z_1$, a contradiction. Thus we assume that
$y$ is not contained in a triangle and $N(y)$ is an independent set.
Similarly, let $z$ be a vertex with distance 2 from $x$, and $y$ be
a common neighbor of $x$ and $z$. If $z$ is contained in a triangle
$zz'z''$, then clearly $yz',yz''\notin E(G)$, otherwise $y$ will be
contained in a triangle. Thus the subgraph induced by
$\{z,z',z'',y\}$ is a $Z_1$, a contradiction. Thus we assume that
$z$ is not contained in a triangle and $N(z)$ is an independent set.
Thus we have that every vertex adjacent to $x$ or with distance 2
from $x$ is not contained in a triangle.

Let $w$ be a neighbor of $x$ in $C$ and $y$ be a neighbor of $x$ in
$D$. Let $y'$ be a neighbor of $x$ in $D$ other than $y$. Since
$yy'\notin E(G)$, by Lemma 1, we have that $d(y)+d(y')\geq n-1$.
Without loss of generality, we assume that $d(y)\geq(n-1)/2$. Note
that $x$ and $y$ have no common neighbors, we have
$d(x)\leq(n+1)/2$.

\noindent\textbf{Case A.} $d(x)=(n+1)/2$.

In this case, $n$ is odd. Let $Y=N(x)\setminus\{w\}$ and
$Z=V(G)\setminus Y\setminus\{x,w\}$. Then $|Y|=(n-1)/2$ and
$|Z|=(n-3)/2$. Since $d(y)\geq(n-1)/2$ and $y$ is nonadjacent to any
vertex in $Y\cup\{w\}$, we have that $y$ is adjacent to every vertex
in $Z$ and $d(y)=(n-1)/2$. This implies that $Z\subset V(D)$. Thus
every vertex in $N_C(x)$ will have degree 1. This implies that $w$
is the only vertex in $C$ and $Y\subset V(D)$.

Note that $d(y)=(n-1)/2$. Let $y'$ be a vertex in $Y$ other than
$y$. By Lemma 1, $d(y)+d(y')\geq n-1$. Thus $d(y')\geq(n-1)/2$.
Since $y'$ is nonadjacent to any vertices in $Y\cup\{w\}$, $y'$ is
adjacent to every vertex in $Z$. This implies that every vertex in
$Y$ and every vertex in $Z$ are adjacent.

Let $Y=\{y_1,y_2,\ldots,y_{(n-1)/2}\}$ and
$Z=\{z_1,z_2,\ldots,z_{(n-3)/2}\}$. Then $P=wxy_1z_1y_2z_2\cdots$
$z_{(n-3)/2}y_{(n-1)/2}$ is a Hamilton path of $G$.

\noindent\textbf{Case B.} $d(x)=n/2$.

In this case, $n$ is odd and $d(y)\geq n/2$. Let
$Y=N(x)\setminus\{w\}$ and $Z=V(G)\setminus Y\setminus\{x,w\}$. Then
$|Y|=(n-2)/2$ and $|Z|=(n-2)/2$. Since $d(y)\geq n/2$ and $y$ is
nonadjacent to any vertices in $Y\cup\{w\}$, we have that $y$ is
adjacent to every vertices in $Z$ and $d(y)=n/2$. This implies that
$Z\subset V(D)$. Thus every vertex in $N_C(x)$ will have degree 1
and then $w$ is the only vertex in $C$ and $Y\subset V(D)$. Note
that $d(x)\geq 3$, $n\geq 6$ and $|Z|\geq 2$.

Let $Y=\{y_1,y_2,\ldots,y_{(n-2)/2}\}$, where $y_1$ is with the
minimum degree and $Z=\{z_1,z_2,$ $\ldots,z_{(n-2)/2}\}$ where $z_1$
is with the maximum degree. For every vertex $y_i$ in $Y$ other than
$y_1$, since $d(y_1)+d(y_i)\geq n-1$, we have $d(y_i)\geq n/2$.
Since $y_i$ is nonadjacent to any vertices in $Y\cup\{w\}$, $y_i$ is
adjacent to every vertex of $Z$. This implies that every vertex in
$Y\setminus\{y_1\}$ and every vertex in $Z$ are adjacent to each
other.

Let $z_i$ be a vertex of $Z$ other than $z_1$. Then the subgraphs
induced by $\{y,x,z_1,z_i\}$ is a claw. Since $d(x)=n/2$, we have
that $d(z_1)\geq(n-2)/2$. Note that $z_1$ is nonadjacent to any
vertex in $Z\cup\{x,w\}$, we have that $z_1$ is adjacent to every
vertex in $Y$ and then $y_1z_1\in E(G)$.

Thus $P=wxy_1z_1y_2z_2\cdots$ $y_{(n-2)/2}z_{(n-2)/2}$ is a Hamilton
path of $G$.

\noindent\textbf{Case C.} $d(x)\leq(n-1)/2$.

Note that $d(x)\geq 3$, we have $n\geq 7$ and $d(y)\geq(n-1)/2\geq
3$. Let $z$ be a neighbor of $y$ other than $x$ with the maximum
degree. Let $z'$ be a neighbor of $y$ other than $x$ and $z$. Then
the subgraph induced by $\{y,x,z,z'\}$ is a claw. Since
$d(x)\leq(n-1)/2$, we have that $d(z)\geq(n-1)/2$.

Let $Y=N(z)$ and $Z=V(G)\setminus Y\setminus\{x,w\}$. Note that
$d(y)\geq(n-1)/2$ and $y$ is nonadjacent to any vertices in
$Y\cup\{w\}$ and $d(z)\geq(n-1)/2$ and $z$ is nonadjacent to any
vertices in $Z\cup\{x,w\}$. We have that $|Y|=(n-1)/2$,
$|Z|=(n-3)/2$ and $y$ is adjacent to every vertex in $Z$. This
implies that there is only the one vertex $w$ in $C$ and $Y,Z\subset
V(D)$.

Note that $x$ has at least two neighbors in $Y$. Let
$Y=\{y_1,y_2,\ldots,y_{(n-1)/2}\}$, where $y_1$ and $y_2$ are two
neighbors of $x$ and $Z=\{z_1,z_2,$ $\ldots,z_{(n-3)/2}\}$ where
$z_1$ is with the minimum degree. Since $d(y_1)+d(y_2)\geq n-1$ and
$y_1$ and $y_2$ are nonadjacent to any vertices in $Y\cup\{w\}$, we
have that $y_1$ and $y_2$ are adjacent to any vertices in $Z$, and
then $y_1z_1,y_2z_1\in E(G)$.

Let $z_i$ be a vertex of $Z$ other than $z_1$. Then the subgraphs
induced by $\{y,x,z_1,z_i\}$ is a claw. Since $d(x)\leq(n-1)/2$, we
have that $d(z_i)\geq(n-1)/2$. Note that $z_i$ is nonadjacent to any
vertices in $Z\cup\{x,w\}$, we have that $z_i$ is adjacent to every
vertex in $Y$. This implies that every vertices in $Y$ and every
vertex in $Z\setminus\{z_1\}$ are adjacent.

Thus $P=wxy_1z_1y_2z_2\cdots$ $z_{(n-3)/2}y_{(n-1)/2}$ is a Hamilton
path of $G$.

\begin{case}
$G$ is 2-connected.
\end{case}

Let $P=v_1v_2\cdots v_p$ be a longest path of $G$. Assume that $G$
is not traceable. Then $V(G)\backslash V(P)\neq \emptyset$. Since
$G$ is 2-connected, there exists a path $R$ with two end-vertices in
$P$ and of length at least 2 which is internally disjoint with $P$.
Let $R=x_0x_1x_2\cdots x_{r+1}$, where $x_0=v_i$ and $x_{r+1}=v_j$,
be such a path as short as possible. Clearly $i\neq 1,p$ and $j\neq
1,p$. Without loss of generality, we assume that $2\leq i<j\leq
p-1$.

\begin{claim}
Let $x\in V(R)\setminus\{v_i,v_j\}$ and $y\in
\{v_{i-1},v_{i+1},v_{j-1},v_{j+1}\}$. Then $xy\notin
\widetilde{E}_{-1}(G)$.
\end{claim}

\begin{proof}
Without loss of generality, we assume $y=v_{i-1}$. If
$xv_{i-1}\in\widetilde{E}_{-1}(G)$, then
$P'=P[v_1,v_{i-1}]v_{i-1}xR[x,v_i]$ $v_iP[v_i,v_p]$ is an
$o_{-1}$-path which contains all the vertices in $V(P)\cup
V(R[x,v_i])$. By Lemma 2, there is a path containing all the
vertices in $P'$, a contradiction.
\end{proof}

\begin{claim}
$v_{i-1}v_{i+1}\in\widetilde{E}_{-1}(G)$,
$v_{j-1}v_{j+1}\in\widetilde{E}_{-1}(G)$.
\end{claim}

\begin{proof}
If $v_{i-1}v_{i+1}\notin E(G)$, by Claim 1, the graph induced by
$\{v_i,x_1,v_{i-1},v_{i+1}\}$ is a claw, where $d(x_1)+d(u_{i\pm
1})<n-1$. Since $G$ is a claw-$o_{-1}$-heavy graph, we have that
$d(v_{i-1})+d(v_{i+1})\geq n$.

The second assertion can be proved similarly.
\end{proof}

\begin{claim}
$v_{i-1}v_{j-1}\notin\widetilde{E}_{-1}(G)$,
$v_{i+1}v_{j+1}\notin\widetilde{E}_{-1}(G)$.
\end{claim}

\begin{proof}
If $v_{i-1}v_{j-1}\in\widetilde{E}_{-1}(G)$, then
$P'=P[v_1,v_{i-1}]v_{i-1}v_{j-1}P[v_{j-1},v_i]v_iRv_jP[v_j,v_p]$ is
an $o_{-1}$-path which contains all the vertices in $V(P)\cup V(R)$,
a contradiction.

The second assertion can be proved similarly.
\end{proof}

\begin{claim}
Either $v_{i-1}v_{i+1}\in E(G)$ or $v_{j-1}v_{j+1}\in E(G)$
\end{claim}

\begin{proof}
Assume the opposite. By Claim 2 we have $d(v_{i-1})+d(v_{i+1})\geq
n-1$ and $d(v_{j-1})+d(v_{j+1})\geq n-1$. By Claim 3, we have
$d(v_{i-1})+d(v_{j-1})<n-1$ and $d(v_{i+1})+d(v_{j+1})<n-1$, a
contradiction.
\end{proof}

Without loss of generality, we assume that $v_{i-1}v_{i+1}\in E(G)$.
Then the subgraph induced by $\{v_i,v_{i-1},v_{i+1},x_1\}$ is a
$Z_1$, a contradiction.

The proof is complete.

\section{Proof of Theorem 8}

If $G$ contains only one or two vertices, then the result is
trivially true. So we assume that $G$ contains at least three
vertices. We use $n$ to denote the order of $G$. We distinguish two
cases.

\setcounter{case}{0}
\begin{case}
$G$ is separable.
\end{case}

Let $x$ be a cut-vertex of $G$. By Lemma 1, $G-x$ has exactly two
components. Let $C$ and $D$ be the two components of $G-x$.

If there is a vertex in $D$ which is nonadjacent to $x$, then let
$z$ be a vertex in $D$ with distance 2 from $x$, and $y$ be a common
neighbor of $x$ and $z$. Let $w$ be a neighbor of $x$ in $C$. Then
$wxyz$ is an induced $P_4$ of $G$, a contradiction. Thus we have
that $x$ is adjacent to every vertex in $D$. By Lemma 1, for every
two vertices $y$ and $y'$ in $D$, $yy'\in\widetilde{E}_{-1}(G)$.
Similarly, $x$ is adjacent to every vertex in $C$ and for every two
vertices $w$ and $w'$ in $C$, $ww'\in\widetilde{E}_{-1}(G)$.

Let $V(C)=\{w_1,w_2,\ldots,w_k\}$ and $V(D)=\{y_1,y_2,\ldots,y_l\}$,
where $k+l+1=n$. Then $P'=w_1w_2\cdots w_kxy_1y_2\cdots y_l$ is an
$o_{-1}$-path of $G$. By Lemma 2, there is a path $P$ containing all
the vertices in $P'$, which is a Hamilton path.

\begin{case}
$G$ is 2-connected.
\end{case}

Let $P=v_1v_2\cdots v_p$ be a longest path of $G$. Assume that $G$
is not traceable. Then $V(G)\backslash V(P)\neq \emptyset$. Since
$G$ is 2-connected, there exists a path $R$ with two end-vertices in
$P$ and of length at least 2 which is internally disjoint with $P$.
Let $R=x_0x_1x_2\cdots x_{r+1}$, where $x_0=v_i$ and $x_{r+1}=v_j$.
Clearly $i\neq 1,p$ and $j\neq 1,p$. Without loss of generality, we
assume that $2\leq i<j\leq p-1$.

Similar as in Section 5, we can prove that

\setcounter{claim}{0}
\begin{claim}
Let $x\in V(R)\setminus\{v_i,v_j\}$ and $y\in
\{v_{i-1},v_{i+1},v_{j-1},v_{j+1}\}$. Then $xy\notin
\widetilde{E}_{-1}(G)$.
\end{claim}

\begin{claim}
$v_{i-1}v_{i+1}\in\widetilde{E}_{-1}(G)$,
$v_{j-1}v_{j+1}\in\widetilde{E}_{-1}(G)$.
\end{claim}

Now we prove that

\begin{claim}
$v_iv_{j-1}\notin E(G)$.
\end{claim}

\begin{proof}
If $v_iv_{j-1} E(G)$, then
$P'=P[v_1,v_{i-1}]v_{i-1}v_{i+1}P[v_{i+1},v_{j-1}]v_{j-1}v_iRv_jP[v_j,v_p]$
is an $o_{-1}$-path containing all the vertices in $V(P)\cup V(R)$,
a contradiction.
\end{proof}

Let $v_k$ be the first vertex in $P[v_{i+1},v_{j-1}]$ which is
nonadjacent to $v_i$. We have that $i+2\leq k\leq j-1$.

\begin{claim}
$x_1v_{k-1}\notin E(G)$, $x_1v_k\notin E(G)$.
\end{claim}

\begin{proof}
If $v_{k-1}=v_{i+1}$, then by Claim 1, we have $x_1v_{i+1}\notin
E(G)$. If $i+2\leq k-1\leq j-2$ and $x_1v_{k-1}\in E(G)$, then
$P'=P[v_1,v_{i-1}]v_{i-1}v_{i+1}P[v_{i+1},v_{k-2}]v_{k-2}v_ix_1v_{k-1}P[v_{k-1},v_p]$
is an $o_{-1}$-path containing all the vertices in $V(P)\cup V(R)$,
a contradiction. Thus we have that $x_1v_{k-1}\notin E(G)$.

If $z_1v_k\in E(G)$, then
$P'=P[v_1,v_{i-1}]v_{i-1}v_{i+1}P[v_{i+1},v_{k-1}]v_{k-1}v_ix_1v_kP[v_k,v_p]$
is an $o_{-1}$-path containing all the vertices in $V(P)\cup V(R)$,
a contradiction. Thus we have that $x_1v_k\notin E(G)$.
\end{proof}

Thus $x_1v_iv_{k-1}v_k$ is an induced $P_4$, a contradiction.

The proof is complete.

\section{Remark}

Here we explain why we use the concept $o_{-1}$-heavy. In fact one
can similarly define $o_r$-heavy subgraphs for an integer $r$.

Let $G$ be a graph on $n$ vertices, $G'$ be an induced subgraph of
$G$ and $r$ be a given integer. We say that $G'$ is
$o_r$-\emph{heavy} if there are two nonadjacent vertices in $V(G')$
with degree sum at least $n+r$. For a given graph $H$, the graph $G$
is called $H$-$o_r$-\emph{heavy} if every induced subgraph of $G$
isomorphic to $H$ is $o_r$-heavy. For a family $\mathcal{H}$ of
graphs, $G$ is called $\mathcal{H}$-$o_r$-\emph{heavy} if $G$ is
$H$-$o_r$-heavy for every $H\in\mathcal{H}$. Clearly, an $H$-free
graph is $H$-$o_r$-heavy for any integer $r$; and if $r\leq s$, then
an $H$-$o_s$-heavy graph is also $H$-$o_r$-heavy.

Consider the bipartite graph $K_{k,k+2}$. Note that every subgraph
of $K_{k,k+2}$ (other than $K_1$ and $K_2$) is $o_{-2}$-heavy and
$K_{k,k+2}$ is non-traceable. Thus, for any class $\mathcal {H}$ of
graphs, a connected $\mathcal{H}$-$o_r$-heavy graph is not
necessarily traceable for $r\leq -2$.

Now we consider the $o_r$-heavy subgraph conditions for $r\geq 0$.
We will show that we cannot get any other pares of subgraphs solving
our problem in Theorems 6 and 9.

\begin{theorem}
Let $r\geq 0$ be an integer. Let $R$ and $S$ be connected graphs
with $R,S\neq P_3$ and let $G$ be a connected graph. Then $G$ being
$\{R,S\}$-$o_r$-heavy implies $G$ is traceable if and only if (up to
symmetry) $R=K_{1,3}$ and $S=C_3$.
\end{theorem}

\begin{theorem}
Let $r\geq 0$ be an integer. Let $S$ be a connected graph with
$S\neq P_3$ and let $G$ be a connected claw-$o_r$-heavy graph. Then
$G$ being $S$-free implies $G$ is traceable if and only if
$S=C_3,Z_1$ or $P_4$.
\end{theorem}

The necessity of these two theorems is deduced by Theorems 6 and 9
immediately. Here we show the sufficiency of them. In Fig. 2, take
$r\geq 3$ and $n\geq 4k+r-2$ in $G_1$, and take $k\geq r+6$ and
$n\geq 6k+r+10$ in $G_2$. Then $G_1$ is
$\{K_{1,3},P_4\}$-$o_r$-heavy and $G_2$ is
$\{K_{1,3},Z_1\}$-$o_r$-heavy. Since the two graphs are both
non-traceable, we get the sufficiency of Theorems 10 and 11.

\end{document}